\documentclass[11pt,letterpaper]{amsart}
\usepackage{amssymb}
\usepackage{bm}
\usepackage[utf8]{inputenc}
\usepackage{amsthm}
\usepackage{amsmath}
\usepackage{amsfonts}
\usepackage{color}
\usepackage{latexsym}
\usepackage{geometry}
\usepackage{enumerate}
\usepackage[shortlabels]{enumitem}
\usepackage{csquotes}
\usepackage{graphicx}
\usepackage{comment}
\usepackage{appendix}
\usepackage{MnSymbol} 
\usepackage{hyperref}
\hypersetup{
	colorlinks=true,
	linkcolor=blue,
	filecolor=magenta,      
	urlcolor=cyan,
}
\usepackage{bm}
\usepackage{amssymb}
\usepackage{MnSymbol}

\emergencystretch=\maxdimen
\DeclareMathOperator{\Rm}{Rm}
\DeclareMathOperator{\Ric}{Ric}
\DeclareMathOperator{\Vol}{Vol}

\makeatletter
\newcommand*{\rom}[1]{\rm {\expandafter\@slowromancap\romannumeral #1@}}
\makeatother

\def\XXint#1#2#3{{\setbox0=\hbox{$#1{#2#3}{\int}$ }
		\vcenter{\hbox{$#2#3$ }}\kern-.6\wd0}}

\protected\def\vts{%
	\ifmmode
	\mskip0.5\thinmuskip
	\else
	\ifhmode
	\kern0.08334em
	\fi
	\fi
}

\numberwithin{equation}{section}
\newtheorem{Theorem}{Theorem}[section]
\newtheorem{Proposition}[Theorem]{Proposition}
\newtheorem{Lemma}[Theorem]{Lemma}
\newtheorem{Remark}[Theorem]{Remark}
\newtheorem{Corollary}[Theorem]{Corollary}
\theoremstyle{definition}
\newtheorem{Definition}[Theorem]{Definition}

\def \tl {\overline{\ell}}

\begin{document}
	\title[]
	{An $\epsilon$-regularity theorem for Perelman's reduced volume}
	
	\author{Liang Cheng,  Yongjia Zhang}

	\date{}
	

	
	\thanks{Liang Cheng's Research partially supported by  
	National Natural Science Foundation of China 12171180}
	
	\thanks{Yongjia Zhang's research is  supported by National Natural Science Foundation of China NSFC12301076 and Shanghai Sailing Program 23YF1420400.}
	
	\address{School of Mathematics and Statistics  $\&$ Key Laboratory of Nonlinear Analysis and Applications (Ministry of Education), Central  China Normal University, Wuhan, 430079, P.R.China}

	\email{chengliang@ccnu.edu.cn }

\address{School of Mathematical Sciences, Shanghai Jiao Tong University, Shanghai, China, 200240}	
		\email{sunzhang91@sjtu.edu.cn }
	\maketitle

	\begin{abstract}
In this article, we prove an $\epsilon$-regularity theorem for Perelman's reduced volume. We show that on a Ricci flow, if Perelman's reduced volume is close to $1$, then the curvature radius at the base point cannot be too small.
	\end{abstract}

	\section{Introduction}
	The $\epsilon$-regularity theorem is an important tool in geometric analysis, primarily used to study the regularity (smoothness) of the solutions to nonlinear partial differential equations and geometric variational problems. The idea of the $\epsilon$-regularity theorems is that if a certain geometric quantity is sufficiently small in a small region, then the solution in that region is regular in some sense (e.g., smooth).  The $\epsilon$-regularity theorems are particularly useful in the study of minimal surfaces, harmonic maps, Yang-Mills fields, and related problems. 
	
	For geometric flows, many $\epsilon$-regularity theorems are derived through the utilization of monotonicity formulas. The first $\epsilon$-regularity theorem for the Ricci flow is Perelman's pseudolocality theorem \cite{P1}, one may see \cite{tw}\cite{w2}\cite{Bam20c} for recent developments in that respect. These  pseudolocality theorems were obtained by using one of Perelman's monotonicity formulas --- the $\mathcal{W}$-functional.
    
    Another relevant result concerning the $\mathcal{W}$-functional is Hein-Naber's $\epsilon$-regularity theorem \cite{HN}: for each $C>0$ there exists an $\epsilon=\epsilon(n, C)>0$ such that the following holds. Let $\left(M^n, g_t\right)_{t\in I}$ be a Ricci flow with bounded curvature within each compact time interval. Assume that there are a space-time piont $(x,t)\in M\times I$ and a scale $r>0$ with $[t-r^2,t]\subset I$, satisfying
    \begin{align*}
        R_{g_{t-r^2}}\ge -\frac{C}{r^2}, \quad \inf_{\tau\in(0,2r^2)}\mu(g_{t-r^2},\tau)\ge -C,
    \end{align*}
    and 
$$
\mathcal{W}_{x,t}(r^2) \geq-\epsilon,
$$
then we have
$$
r_{\mathrm{Rm}}\left(x,t\right)^2 \geq \epsilon r^2.
$$
Here $\mathcal{W}_{x,t}(\tau) $ is Perelman's $\mathcal{W}$-functional with its test function being the conjugate heat kernel based at $(x,t)$ (also known as Perelman's pointed entropy), $\mu$ is Perelman's $\mu$-functional, and $r_{\mathrm{Rm}}$ is the curvature radius defined in Definition \ref{def:def of r}. 

Bamler \cite{Bam20a} proved a stronger $\epsilon$-regularity theorem for the Nash entropy: 
there is a constant $\epsilon(n)>0$ with the following property. Let $\left(M^n, g_t\right)_{t\in I}$ be a Ricci flow with bounded curvature within each compact time interval. Let $(x,t)\in M\times I$ be a space-time point and  $r>0$ a scale with $[t-r^2,t]\subset I$. If
$$\mathcal{N}_{x, t}\left(r^2\right) \geq-\epsilon,$$ then $$r_{\mathrm{Rm}}(x, t) \geq \epsilon r.$$    
    
 Monotonicity formulas of geometric flows are closely related to their rigidity properties. Here by rigidity we mean that if a geometric quantity is \emph{equal} to the standard value, then the geometric object is \emph{identical} to the standard model. For instance, if Perlman's pointed entropy or the Nash entropy is ever equal to zero, then the Ricci flow is the (static) Euclidean space. In view of this fact, $\epsilon$-regularity theorems and rigidity theorems often appear in pairs, because one monotonicity formula would give rise to both. Indeed, these rigidity properties are essential to the proofs of the aforementioned
 $\epsilon$-regularity theorems. 
 
Another important monotonicity formula discovered by Perelman is the reduced volume, which also has the rigidity property that if the reduced volume is ever equal to 1, then the Ricci flow is the trivial one on the Euclidean space. Therefore, it is reasonable to expect an $\epsilon$-regularity theorem for the reduced volume. The object of the present paper is to prove such a theorem. All definitions and notations in our results can be found in Section 2.1.

\begin{Theorem}\label{main_thm}
 There is a dimensional constant $\epsilon(n)>0$ with the following property. Let
$(M^n, g_t)_{t\in I}$ be a complete  Ricci flow with bounded curvature within each compact time interval. Let $(x,t)\in M\times I$ be a space-time point and $r>0$ a scale with $[t-2r^2,t]\subset I$. If, furthermore
\begin{equation}\label{small_reduced_volume}
\mathcal{V}_{x,t}
(r^2) \ge 1-\epsilon,
\end{equation}
then we have $r_{\operatorname{Rm}} (x, t) \ge  r$; here $\mathcal{V}$ and $r_{\Rm}$ are defined in \eqref{eq:def of V} and Definition \ref{def:def of r}, respectively.
\end{Theorem}

\begin{Remark}
  Yokota \cite{YT} proved a gap theorem for the reduced volume: if an ancient Ricci flow with Ricci curvature bounded from below satisfies
	\begin{align*}
	\lim_{\tau\to\infty}\mathcal{V}_{x,t}(\tau)\ge 1-\epsilon(n),
	\end{align*}
	where $(x,t)$ is a space-time point and $\epsilon(n)$ is a dimensional constant, then the Ricci flow is a static Euclidean space. The second author \cite{Z21} also proved a similar gap theorem for the asymptotic $\mathcal{W}$-entropy on ancient Ricci flows.
    Theorem \ref{main_thm} should be viewed as a complement of \cite{YT}.
\end{Remark}

Let us recall the ideas of the proofs of \cite{HN} and \cite{Bam20a}; each proof breaks down to the following steps:
\begin{enumerate}[(1)]
\item Assume the $\epsilon$-regularity theorem is false, so one can find a sequence of counterexamples --- the Nash entropies go to zero, but the curvature radii are small.
\item Adjust the sequence by a point-picking process, so that the sequence still satisfies the conditions of step (1), and at the same time the geometry is locally uniformly bounded.
\item Apply the compactness theorem, such as \cite{RH5}, to take a smooth limit and obtain a contradiction.
\end{enumerate}

In fact, many $\epsilon$-regularity theorems are proved in the above way. However, when one applies the above idea to Perelman's reduced volume, one would encounter a difficulty in Step (2). The reason is that, unlike the Nash entropy, there is no estimate showing how Perelman's reduced volume depends on the base point in the general case. Nevertheless, the success of Yokota's gap theorem \cite{YT} was due to the fact that, when considering the so-called asymptotic reduced volume 
$$\overline{\mathcal{V}}(x,t)=\lim_{\tau\to\infty}\mathcal{V}_{x,t}(\tau)$$
on an ancient Ricci flow with Ricci curvature bounded from below, and, in particular, the way it depends on the base point $(x,t)$, the fuzzy stuff introduced by local geometry will be eliminated by taking the limit in $\tau$, and hence one can obtain
\begin{align*}
\overline{\mathcal{V}}(x,t) \le \overline{\mathcal{V}}(y,s)\quad \text{whenever}\quad t>s.
\end{align*}
With this estimate, the point-picking argument is clearly available. In fact, this is also the idea of the proof of \cite{Z21}.

However, in the proof of Theorem \ref{main_thm}, we will not perform a point-picking argument, but will overcome this difficulty with much more sophisticated techniques --- Bamler's $\mathbb{F}$-compactness and partial regularity theories. Namely, in Step (2), we will perform a simple normalization instead of point-picking, and we will take an $\mathbb{F}$-limit in Step (3), instead of smooth limit, to obtain the contradiction.

 It is also straightforward to observe that, by replacing the reduced volume by Nash entropy and applying the similar ``no-point-picking-argument'' of the current article, Bamler's $\epsilon$-regularity theorem \cite[Theorem 10.2]{Bam20a} can be improved: if $[t-2r^2,t]\subset I$ and $$\mathcal{N}_{x,t}(r^2)\ge -\epsilon(n),$$ then $$r_{\Rm}(x,t)\ge r;$$
 here (and in Theorem \ref{main_thm} as well) $r_{\Rm}$ can even be the curvature radius defined as \cite[Definition 10.2]{Bam20a}. 
Since the proof is a simple adaptation of our argument, we leave the details to the reader.

Finally, we give a corollary of  Theorem \ref{main_thm}.
\begin{Corollary}\label{coro_main}
For any $\epsilon>0$, there is a positive number $\delta=\delta(\epsilon,n)>0$ with the following property. Let
    $(M^n, g_t)_{t\in I}$ be a complete  Ricci flow with bounded curvature within each compact time interval. Let $(x,t)\in M\times I$ be a space-time point and $r>0$ a scale with $[t-r^2,t]\subset I$. If, furthermore,
\begin{equation}\label{small_reduced_volume}
\mathcal{V}_{x,t}
(r^2) \ge 1-\delta,
\end{equation}
then, for all $r'\in(0,\epsilon^{-1}r)$, we have
\begin{align*}
\operatorname{Vol}_{g_t}\big(B_{g_t}(x,r')\big)\ge (1-\epsilon)\omega_n (r')^n,
\end{align*}
where $\omega_n$ is the volume of the unit ball in $\mathbb{R}^n$.
\end{Corollary}

	\section{Preliminaries }

	\subsection{Definitions}
	
	Let $(M^n,g_t)_{t\in I}$ be a solution to the Ricci flow
\begin{align*}
	\frac{\partial g_t}{\partial t}=-2\Ric_{g_t}.
\end{align*}
Fixing a $t\in I$, Perelman's	$\mathcal{L}$-energy for a piecewise $C^1$ curve $\gamma(\sigma):[0,\tau]\to M$, where $[t-\tau,\tau]\subset I$, is defined as
\begin{align}\label{l-energy} \mathcal{L}(\gamma)=\int^{\tau}_0
	\sqrt{\sigma}\big(R_{g_{t-\sigma}}(\gamma(\sigma))+|\gamma'(\sigma)|_{g_{t-\sigma}}^2\big)d\sigma.
\end{align}	
The critical points of the $\mathcal{L}$-energy are called minimal $\mathcal{L}$-geodesics. Similar to the standard theory of geodesics in Riemannian geometry, the minimizer with respect to fixed end points always exists and is almost always unique, should the Ricci flow in question have a Ricci curvature lower bound \cite{Y1}.

Thus, $\mathcal{L}$-geodesics can be used to define a type of space-time distance, namely, \emph{Perelman's reduced distance}. Let $(x,t)$, $(y,s)\in M\times I$ be two space-time points in the Ricci flow, where $s<t$. Define
\begin{align*}
L_{x,t}(y,t-s):=\inf_{\gamma}\mathcal{L}(\gamma)=\inf_{\gamma}\int_0^{t-s}\sqrt{\tau}\left(R_{g_{t-\tau}}(\gamma(\tau))+|\gamma'(\tau)|^2_{g_{t-\tau}}\right)d\tau,
\end{align*}
where the infimum is taken among all piecewise $C^1$-curves $\gamma:[0,t-s]\to M$ with $\gamma(0)=x$, $\gamma(t-s)=y$, then
\begin{align*}
\ell_{x,t}(y,t-s):=\frac{1}{2\sqrt{t-s}}L_{x,t}(y,t-s)
\end{align*} 
is the \emph{reduced distance from $(x,t)$ to $(y,s)$}, and a minimizer $\gamma$ is called a minimizing $\mathcal{L}$-geodesic from $(x,t)$ to $(y,s)$. Furthermore, if we fix $(x,t)$, then the space-time function
\begin{align*}
\ell_{x,t}(\,\cdot\, ,\, \cdot\,): M\times(0,t-\inf I) \to \mathbb{R}
\end{align*}
is called the \emph{reduced distance based at $(x,t)$}. We remark here that the second variable of $\ell_{x,t}(\,\cdot\, ,\, \cdot\,)$ is the backward time starting from $t$. However, sometimes it is also convenient to consider the forward time. So we define
\begin{align}\label{eq:def of tl}
\tl_{x,t}(y,s):=\ell_{x,t}(y,t-s),\quad (y,s)\subset M\times \big(I\cap(-\infty,t)\big).
\end{align}

Perelman's \emph{reduced volume} is defined as 
\begin{align}\label{eq:def of V}
\mathcal{V}_{x,t}(\tau):=\int_M (4\pi\tau)^{-\frac{n}{2}}e^{-\ell_{x,t}(\cdot,\tau)}\,dg_{t-\tau},
\end{align}
where $(x,t)$ is the base point and $\tau\ge 0$ is the backward time from $t$. The reduced volume is one of the two important monotonic quantities discovered by Perelman --- $\mathcal{V}_{x,t}(\tau)$ is always decreasing with respect to $\tau$.

\begin{Definition}\label{def:def of r}
Let $\left(M^n, g_t\right)_{t\in I}$ be a complete Ricci flow. Given $(x, t) \in M \times I$, the curvature radius at $(x,t)$ is defined as
	$$
	r_{\mathrm{Rm}}(x, t) := \sup \left\{r>0\,\bigg|\, [t-r^2,t]\subset I,\ \sup _{P_r(x, t)}|\mathrm{Rm}| \leq r^{-2}\right\},
	$$
where 	
$$
P_r(x, t) := B_{g_t}(x, r) \times\left[t-r^2, t\right].
$$
\end{Definition}

\subsection{Perelman's $\mathcal{L}$-geometry}	 Consider a Ricci flow $(M,g_t)_{t\in I}$ and fix $(x,t) \in M\times I$. We shall recall some basic properties of $\ell_{x,t}$. By the first variation of the $\mathcal{L}$-energy \eqref{l-energy}, the $\mathcal{L}$-geodesic equation is (see \cite[(7.2)]{P1})
\begin{eqnarray}\label{l-geodesic}
	\nabla^{g_{t-\tau}}_{\gamma'}\gamma'(\tau)-\frac{1}{2}\nabla^{g_{t-\tau}} R_{g_{t-\tau}}(\gamma(\tau))+\frac{1}{2\tau}\gamma'(\tau)+2\Ric_{g_{t-\tau}}(\gamma'(\tau))=0.
\end{eqnarray}	
Given any $v\in T_xM$, denote by $\gamma_v$ the $\mathcal{L}$-geodesic satisfying $\lim\limits_{\sigma\to 0}\sqrt{\sigma}\gamma'(\sigma) = v$. Then the $\mathcal{L}$-exponential map $\mathcal{L}\text{exp}^{\tau}_{x,t}:T_xM\to M$ is defined as
	$$\mathcal{L}\text{exp}^{\tau}_{x,t}(v)=\gamma_v(\tau).$$

Similar to the standard theory of geodesics, we define
\begin{align*}
\Omega^{TM}_{x,t}(\tau)=\Big\{ v\in T_xM\ \Big\vert & \ \gamma_v|_{[0,\tau]} :[0,\tau] \to M \text{ is the unique minimizing }
 \mathcal{L}\text{-geodesic}	
\\
& \text{\ from $(x,t)$ to $(\gamma_v(\tau),t-\tau)$}; 
\\
& \text{\ $x$ and $\gamma_v(\tau)$ are not conjugate along }\gamma_v.\Big\}.
\end{align*}
Correspondingly, 	
\begin{align*}
\Omega_{x,t}(\tau)=\Big\{y\in M\ \Big\vert\ &\  \text{There is a unique minimizing } \mathcal{L}\text{-geodesic }\gamma:[0,\tau] \to M 
\\	
&	\text{\ from $(x,t)$ to $(y,t-\tau)$;}\text{ $x$ and $y$ are not conjugate along }  \gamma.\Big\}.
\end{align*}
It is well known that
$$
\Omega_{x,t}(\tau)=\mathcal{L}\text{exp}_{x,t}^{\tau}\big(\Omega_{x,t}^{TM}(\tau)\big),
$$
and the  $\mathcal{L}$-cut-locus is
	defined as
	$$
	C_{x,t}(\tau)=M\backslash \Omega_{x,t}(\tau).
	$$
When the base point is understood, we also omit the subindices in the notations introduced above.
	
	Generalizing Perelman's results in \cite{P1}, Ye \cite{Y1} studied the properties of the $\ell$-function and the $\mathcal{V}$-function assuming only a lower bound for the Ricci curvature. We now state these useful results.

	\begin{Theorem}[Proposition 2.7, Proposition 2.11, and Lemma 2.14 in \cite{Y1}]\label{Y}
Let $(M^n,g_t)_{t\in I}$ be a Ricci flow such that the Ricci curvature of each time-slice is bounded from below. Fixing a base-point $(x,t)\in M\times I$, the following hold:
\begin{enumerate}
\item For any $(y,s)\in M\times I$ with $s<t$, there exists a minimal $\mathcal{L}$-geodesic connecting  $(x,t)$ and $(y,s)$. In other words, for $\tau>0$ with $t-\tau\in I$, the map
		$\mathcal{L}\text{exp}^{\tau}_{x,t}$ is onto. 
\item  $L_{x,t}$ is locally Lipschitz in space-time.
\item For each $\tau$ with $t-\tau\in I$, $C_{x,t}(\tau)\subset M$ is a closed set of zero Remannian measure in $(M,g_{t-\tau})$. Consequently, $\bigcup_{0<\tau,t-\tau\in I}C_{x,t}(\tau)\times \{t-\tau\}$
		is a closed set of zero in space-time.
\end{enumerate}
	\end{Theorem}

Furthermore, we need the following analytic properties of the $\ell$-function. 

	\begin{Theorem}[\cite{P1}, see also Lemma 2.19 and Theorem 2.20 in \cite{Y1}]\label{Perelman}
		Let $(M^n,g_t)_{t\in I}$ be a Ricci flow such that the Ricci curvature of each time-slice is bounded from below. Let $ \ell:=\ell_{x,t}$ be the reduced distance function based at a fixed point $(x,t)$. Then on $ \bigcup_{0<\tau, t-\tau\in I}\Omega_{x,t}(\tau)\times \{\tau\}$ it holds that:
		\begin{align}
		&2\frac{\partial \ell}{\partial \tau}+|\nabla \ell|_{g_{t-\tau}}^2-R_{g_{t-\tau}}+\frac{\ell}{\tau}=0,\label{eq_l_1}\\
		&	\frac{\partial }{\partial \tau}\ell-\Delta_{g_{t-\tau}} \ell +	|\nabla \ell|_{g_{t-\tau}}^2-R_{g_{t-\tau}}+\frac{n}{2\tau} \geq 0,\label{eq_l_4}\\
		&	2\Delta_{g_{t-\tau}} \ell-|\nabla \ell|_{g_{t-\tau}}^2+R_{g_{t-\tau}}+\frac{\ell-n}{\tau} \le 0.\label{eq_l_5}
		\end{align}
Furthermore, (\ref{eq_l_4}) and (\ref{eq_l_5}) both hold in the sense of distribution. That is to say, for any $0<\tau_1<\tau_2$ with $t-\tau_2>\inf I$ and for any nonnegative Lipshcitz function $\phi$ compactly supported on $M\times [\tau_1,\tau_2]$, it holds that
\begin{eqnarray}\label{eq_l_6}
\int_{\tau_1}^{\tau_2}\int_{M}\Bigg(\langle\nabla \ell,\nabla\phi\rangle_{g_{t-\tau}}+\Big(\frac{\partial}{\partial\tau}\ell+|\nabla \ell|_{g_{t-\tau}}^2-R_{g_{t-\tau}}+\frac{n}{2\tau}\Big)\phi\Bigg)d{g_{t-\tau}}d\tau\geq0,
\end{eqnarray}
and, for any $\tau>0$ with $t-\tau>\inf I$ and any nonnegative Lipshcitz function $\phi$ compactly supported on $M$, it holds that
\begin{eqnarray}\label{eq_l_7}
\int_{M}\Bigg(-2\langle\nabla \ell,\nabla\phi\rangle_{g_{t-\tau}}+\Big(-|\nabla \ell|_{g_{t-\tau}}^2+R_{g_{t-\tau}}+\frac{\ell-n}{\tau}\Big)\phi\Bigg)d{g_{t-\tau}}\leq 0.
\end{eqnarray}
	\end{Theorem}

The following is a useful consequence of the first variation formula.
	
\begin{Lemma}[Perelman \cite{P1}] \label{grad_ll}
Let $\gamma$ be an $\mathcal{L}$-geodesic starting from $(x,t)$. Then, so long as $t-\tau>\inf I$ and $\gamma(\tau)\in \Omega_{x,t}(\tau)$, it holds that
\begin{eqnarray}
\nabla^{{g_{t-\tau}}} \ell(\gamma(\tau),\tau)=\gamma'(\tau).
\end{eqnarray}
\end{Lemma}

The following monotonicity formula is also a well-known result of Perelman.

	\begin{Theorem}[Perelman\cite{P1}, see also Theorem 4.3 and Theorem 4.5 in \cite{Y1}]\label{Monotonicity}
Let $(M^n,g_t)_{t\in I}$ be a Ricci flow such that the Ricci curvature of each time-slice is bounded from below. For any $(x,t)\in M\times I$, the reduced volume $\mathcal {V}_{x,t}(\tau)$ satisfies
 \begin{enumerate}
 \item $\mathcal{V}_{x,t}(\tau)\leq 1$ for all $\tau>0$ with $t-\tau> \inf I$.
 \item $\mathcal{V}_{x,t}(\tau)$ is non-increasing in $\tau$.
 \end{enumerate}
	\end{Theorem}

	\subsection{Bamler's theory of non-collpased limits of Ricci flows}
Bamler's definition of metric flow and $\mathbb{F}$-convergence in their full lengths is beyond our scope of exposition here. We will assume of the reader some familiarity with the contents of \cite{Bam20a, Bam20c, Bam20b}, and will only introduce the results that are most pertinent to our work.

Consider a sequence of $n$-dimensional Ricci flows $(M^i,g^i_t,x_i)_{t\in(-T_i,0]}$, each with bounded curvature within compact time intervals, satisfying
\begin{align}\label{non-collapsing_condition}
	\mathcal{N}_{x_i,0}(\tau)\ge -Y
\end{align}
for some $\tau>0$ and $Y>0$. Denote by $d\nu_{x_i,0\,|\, t}=K(x_i,0\,|\,\cdot,t)\,dg_t$ the conjugate heat kernel based at $(x_i,0)$, viewed as an evolving probability measure. Then, according to \cite[Theorem 7.7]{Bam20b},  the sequence of pairs $\{((M^i,g^i_t)_{t\in (-T_i,0]}, (\nu_{x_i,0\,|\, t})_{t\in(-T_i,0]})\}_{i=1}^\infty$, after passing to a subsequence which we shall not relabel,
converges to a metric flow pair $(\mathcal{X},(\nu_t)_{t<0})$ in the $\mathbb{F}$-sense, namely,
\begin{gather}\label{eq:definition of the limiting sequence}
	((M^i,g^i_t)_{t\in (-T_i,0]}, (\nu_{x_i,0\,|\, t})_{t\in(-T_i,0]})\xrightarrow[\quad i\to\infty\quad ]{\mathbb{F}}(\mathcal{X},(\nu_t)_{t<0}).
\end{gather}
The limit $\mathcal{X}$ is a metric flow over the time interval $(-T_\infty,0)$, where $T_\infty=\limsup_{i\to\infty}T_i$; see \cite[\S3.1, \S5]{Bam20b} for the definitions of metric flow and $\mathbb{F}$-convergence. $\nu_t$ is a conjugate heat flow on $\mathcal{X}$. In particular, it can be viewed as an evolving probability measure (c.f. \cite[Definition 3.13]{Bam20b}).

Bamler showed that $\mathcal{X}$ has the decomposition
$$
\mathcal{X}=\mathcal{R} \cupdot \mathcal{S},
$$
where $\mathcal{R}$ and  $\mathcal{S}$ are called the regular part and the singular part, respectively; the regular part $\mathcal{R}$ is  a smooth Ricci flow space-time with a time-dependent Riemannian metric $g_t$; the singular part is negligible, in the sense that its space-time Minkowsky codimension is at least $4$; see \cite{Bam20c}.

Let $\mathfrak{t}: \mathcal{X}\to (-T_\infty,0)$ be the time function on the metric flow $\mathcal{X}$. Then, subsets of the forms
\begin{align*}
\mathcal{X}_t:=\mathcal{X}\cap\mathfrak{t}^{-1}(t),\quad \mathcal{X}_{[t_1,t_2]}:=\mathcal{X}\cap\mathfrak{t}^{-1}([t_1,t_2]),
\end{align*}
are called time-slices and time-slabs, respectively. We shall recall some important facts about the limit metric flow $\mathcal{X}$.

\begin{Theorem}[{\cite[Theorem 2.4]{Bam20c}, \cite{Bam20b},\cite[Theorem 9.12]{Bam20b}}]\label{Regularpart-basic-properties}
	The following are true:
	\begin{enumerate}[(a)]
	\item For any $t\in(-T_\infty,0)$, $\mathcal{S}_t:=\mathcal{S} \cap \mathcal{X}_t$ is a set of measure zero in the sense that $\nu_t(\mathcal{S}_t)=0$.
	\item For any $t\in(-T_\infty,0)$, the time-slice $\mathcal{X}_t$ is a metric space arising as a metric completion of the length metric on $\left(\mathcal{R}_t, g_t\right)$. In other words, $\mathcal{R}_t:=\mathcal{R} \cap \mathcal{X}_t \subset \mathcal{X}_t$ is open and dense, and the metric $d_t$ of $\mathcal{X}_t$, when restricted to $\mathcal{R}_t$, agrees with the length metric of $g_t$.\\
	\item $\mathcal{X}$ is metric flow of full support in the sense of \cite[Definition 3.20]{Bam20b}. In particular, for any $x\in\mathcal{X}$ and $r>0$, let $t=\mathfrak{t}(x)$, we have
	\begin{align*}
	\nu_t(B_t(x,r))>0.
	\end{align*}
	 \item On the regular part $\mathcal{R}$, we have
\begin{align*}
d\nu_t= (4\pi|t|)^{-\frac{n}{2}}e^{-f}\,dg_t,
\end{align*} 
where $(4\pi|t|)^{-\frac{n}{2}}e^{-f}$ is a positive solution to the conjugate heat equation $(-\partial_{\mathfrak{t}}-\Delta_{g_t}+R_{g_t})u=0$ on $\mathcal{R}$.
	\end{enumerate}	
	
\end{Theorem}

It is also crucial to our proof that the convergence in \eqref{eq:definition of the limiting sequence} can be updated to smooth convergence on $\mathcal{R}$. 

\begin{Theorem}[{\cite[Theorem 9.21]{Bam20b}, \cite[Theorem 2.5]{Bam20c}}]\label{Thm:local uniform convergence}
There is an increasing sequence of open sets $U_1\subset U_2\subset \hdots \subset \mathcal{R}$ with $\displaystyle \cup_{i=1}^\infty U_i=\mathcal{R}$, open sets $V_i\subset M^i\times(-T_i,0)$, time-preserving diffeomorphisms $\psi_i:U_i\to V_i$, and a sequence of positive numbers $\delta_i\searrow 0$, such that
\begin{align*}
\left\|\psi_i^*g^i-g\right\|_{C^{[\delta_i^{-1}]}(U_i)} &\ <\delta_i,
\\
\left\|\psi_i^*\partial_{\mathfrak{t},i}-\partial_{\mathfrak{t}}\right\|_{C^{[\delta_i^{-1}]}(U_i)} &\ <\delta_i,
\\
\left\|K(x_i,0\,|\,\cdot, t)\circ\psi_i- (4\pi|t|)^{-\frac{n}{2}}e^{-f}\right\|_{C^{[\delta_i^{-1}]}(U_i)} &\ <\delta_i,
\end{align*}
where $K$ is the conjugate heat kernel and $d\nu_t= (4\pi|t|)^{-\frac{n}{2}}e^{-f}\,dg_t$ on $\mathcal{R}$.
\end{Theorem}

\section{local estimates for the reduced distance}

	\begin{Proposition}\label{l-estiamtes}
If
\begin{align*}
r>0, \quad L< \infty,\quad T<\infty, \quad C\ge \underline{C}(r,L,T),
\end{align*}	
	then the following holds.
	
		Let $(M^n,g_t)_{t\in I}$ be a complete Ricci flow with bounded curvature within each compact time interval and let $\ell:=\tl_{x_0,t_0}$ be  (the forward-time version of) the reduced distance based at a fixed space-time point $(x_0,t_0)\in M\times I$; see \eqref{eq:def of tl}. Assume that for a space-time point $(x,t)\in M\times I$, it holds that
		\begin{enumerate}[(a)]
		\item $\displaystyle [t-r^2,t]\subset I$ and $0< t_0-t \le T$;
		\item $\displaystyle |\Rm| \le r^{-2}$ on $\displaystyle B_{g_t}(x,r)\times[t-r^2,t]$;
		\item $\ell(x,t) \le L$.
		\end{enumerate}
Then we have 
\begin{enumerate}
\item $|\ell|\le C$ on $B_{g_t}(x,r)\times [t-r^2/2,t-r^2/8]$;
\item $\displaystyle\left|\tfrac{\partial}{\partial s}\ell\right|+|\nabla\ell|\le C$ on $B_{g_t}(x,r/2)\times [t-r^2/2,t-r^2/4]$.
\end{enumerate}		
	\end{Proposition}
	
\emph{Remark.}  We shall prove that conclusion (2) of the proposition holds on $B_{g_t}(x,r/2)\times [t-r^2/2,t-r^2/4] \cap (\cup_{\tau>0}\Omega_{x_0,t_0}(\tau)\times\{t_0-\tau\})$. However, by the local Lipschitz property of $\ell$ and its absolute continuity, this is sufficient to entail (2) as a local $C^{0,1}$ estimate.

\begin{proof}
(1) First of all, the lower bound of $\ell$ follows easily from the lower bound of the scalar curvature. By the maximum principle, we have
\begin{align*}
R\ge -\frac{n}{r^2}\quad \text{ on }\quad M \times [t-r^2/2,t_0].
\end{align*}
Thus, for any $(y,s)\in M\times [t-r^2/2,t-r^2/8]$, we have
\begin{align}\label{eq:lower C0 of ll}
\ell(y,s) = &\ \frac{1}{2\sqrt{t_0-s}}\inf_\gamma \int_0^{t_0-s}\sqrt{\tau}\big(|\gamma'(\tau)|^2_{g_{t_0-\tau}}+R_{g_{t_0-\tau}}(\gamma(\tau))\big)\,d\tau
\\\nonumber
\ge &\ -\frac{1}{2\sqrt{t_0-s}}\int_0^{t_0-s}\sqrt{\tau}\cdot \frac{n}{r^2}\,d\tau
\\\nonumber
\ge &\ -C(r, T).
\end{align}

For the upper bound, we consider an arbitrary point $(y,s)\in B_{g_t}(x,r)\times [t-r^2/2,t-r^2/8]$. Let $\gamma:[0,t_0-s]\to M$ be the concatenation of two curves:
\begin{align*}
\gamma(\tau)= \left\{\begin{array}{rl}
\zeta_1(\tau) & \tau \in [0,t_0-t]
\\
\zeta_2(\tau) & \tau \in [t_0-t,t_0-s]
\end{array}\right.,
\end{align*}
where $\zeta_1$ is a minimal $\mathcal{L}$-geodesic from $(x_0,t_0)$ to $(x,t)$, and $\zeta_2$ is a minimizing geodesic with respect to $g_t$ connecting $x$ and $y$. Note that $\mathcal{L}(\zeta_1) \le 2\sqrt{t_0-t}L$ and $\zeta_2\subset B_{g_t}(x,r)$, and that $|\zeta_2'(\tau)|_{g_t}=\frac{d_{g_t}(x,y)}{t-s}\le 8r^{-1}$. Thus $\gamma$ is a piecewise $C^1$-curve connecting $(x_0,t_0)$ and $(y,s)$, and we may estimate by the curvature assumption in (b):
\begin{align*}
\mathcal{L}(\gamma)=&\ \int_0^{t_0-s}\sqrt{\tau}\big(|\gamma'(\tau)|^2_{g_{t_0-\tau}}+R_{g_{t_0-\tau}}(\gamma(\tau))\big)\,d\tau
\\
= &\ \mathcal{L}(\zeta_1) + \int_{t_0-t}^{t_0-s}\sqrt{\tau}\big(|\zeta_2'(\tau)|^2_{g_{t_0-\tau}}+R_{g_{t_0-\tau}}(\zeta_2(\tau))\big)\,d\tau
\\
\le &\ C(r,L,T) + \int_{t_0-t}^{t_0-s}\sqrt{\tau}|\zeta_2'(\tau)|^2_{g_{t_0-\tau}} \, d\tau
\\
\le &\ C(r,L,T) + \exp\left(\tfrac{C(n)}{r^2}\cdot r^2\right)\int_{t_0-t}^{t_0-s}\sqrt{\tau}|\zeta_2'(\tau)|^2_{g_t} \, d\tau
\\
\le &\ C(r,L,T).
\end{align*}
Thus, part (1) of the proposition is proved.
\\

(2) Let us fix any point $(y,s)\in B_{g_t}(x,r/2)\times [t-r^2/2,t-r^2/4]$ and let $\gamma: [0,t_0-s]\to M$ be the minimizing geodesic from $(x_0,t_0)$ to $(y,s)$. We also assume, without loss of generality, that $(y,s)$ is not in the $\mathcal{L}$-cut-locus of $(x_0,t_0)$. Define
\begin{align*}
\tau_e=\inf\big\{\tau\in[t_0-t+r^2/8,t_0-s]\,\Big|\, \gamma|_{[\tau,t_0-s]}\subset B_{g_t}(x,3r/4)\big\}.
\end{align*}
Then, similar to \cite{CZ}, we have
\\

\noindent\underline{\textbf{Claim.}} \emph{There is a positive number $\epsilon=\epsilon(r,L,T)>0$, such that $\sqrt{t_0-s}-\sqrt{\tau_e}>\epsilon$.}
\\

\begin{proof}[Proof of the claim]
Assume that $t_0-\tau_e<t-r^2/8$, for otherwise there is nothing to prove. Then for $z:=\gamma(\tau_e)$, we must have $d_{g_t}(x,z)=3r/4$. 

Let us perform a change of variable and consider $\beta(\sigma)=\gamma(\sigma^2)$. Then, by part (1) we have
\begin{align*}
\mathcal{L}(\gamma)=\int_0^{\sqrt{t_0-s}}\left(\tfrac{1}{2}|\beta'(\sigma)|^2_{g_{t_0-\sigma^2}}+2\sigma^2R_{g_{t_0-\sigma^2}}(\beta(\sigma))\right)\,d\sigma \le C(r, L,T).
\end{align*}
Consequently, we have
\begin{align*}
\frac{1}{2}\int_{\sqrt{\tau_e}}^{\sqrt{t_0-s}}|\beta'(\sigma)|^2_{g_{t_0-\sigma^2}}\,d\sigma =&\ \mathcal{L}(\gamma)-\mathcal{L}(\gamma|_{[0,\tau_e]}) - \int_{\sqrt{\tau_e}}^{\sqrt{t_0-s}}2\sigma^2R_{g_{t_0-\sigma^2}}(\beta(\sigma))\,d\sigma
\\
\le &\ C(r,L,T),
\end{align*}
where we have applied the curvature assumption (b) and the conclusion of part (1). Arguing in the same way as the proof of part (1), we have
\begin{align*}
\int_{\sqrt{\tau_e}}^{\sqrt{t_0-s}}|\beta'(\sigma)|^2_{g_t}\,d\sigma \le \exp\left(\tfrac{C(n)}{r^2}\cdot r^2\right)\int_{\sqrt{\tau_e}}^{\sqrt{t_0-s}}|\beta'(\sigma)|^2_{g_{t_0-\sigma^2}}\,d\sigma \le C(r,L,T).
\end{align*}

Finally, since $\beta(\sqrt{t_0-s})=y\in B_{g_t}(x,r/2)$, and $\beta(\sqrt{\tau_e})=z$ satisfies $d_{g_t}(x,z)=3r/4$, we have, by the first variation of the geodesic energy:
\begin{align*}
\frac{r^2}{16(\sqrt{t_0-s}-\sqrt{\tau_e})}\le \frac{d^2_{g_t}(y,z)}{\sqrt{t_0-s}-\sqrt{\tau_e}} \le \int_{\sqrt{\tau_e}}^{\sqrt{t_0-s}}|\beta'(\sigma)|^2_{g_t}\,d\sigma  \le C(r,L,T).
\end{align*}
This finishes the proof of the claim.
\end{proof}

Continuing with the proof of part (2), we still consider $\beta(\sigma)=\gamma(\sigma^2)$, the minimizing geodesic connecting $(x_0,t_0)$ and $(y,s)$, then the $\mathcal{L}$-geodesic equation \eqref{l-geodesic} becomes
\begin{equation*}
		\nabla^{g_{t-\sigma^2}}_{\beta'(\sigma)} \beta'(\sigma)-2\sigma^2\nabla^{g_{t-\sigma^2}} R_{g_{t-\sigma^2}}+4\sigma \Ric_{g_{t-\sigma^2}}(\beta'(\sigma))=0,
		\end{equation*}
and so we have
\begin{eqnarray*}
		\frac{d}{d\sigma}|\beta'(\sigma)|_{g_{t_0-\sigma^2}}^2&=&4\sigma^2 \langle\nabla R, \beta'(\sigma)\rangle_{g_{t_0-\sigma^2}}-4\sigma \Ric_{g_{t_0-\sigma^2}}(\beta'(\sigma),\beta'(\sigma))
\\
&\leq& C(r,T)\left(|\beta'(\sigma)|_{g_{t_0-\sigma^2}}+|\beta'(\sigma)|_{g_{t_0-\sigma^2}}^2\right)
		\end{eqnarray*}
for all $\displaystyle\sigma\in\left[\sqrt{\tau_e},\sqrt{t_0-s}\right]$, where we have applied the curvature assumption (b) and Shi's estimate \cite{Sh}. Note that $t-r^2/2\le s<t_0-\tau_e \le t-r^2/8$, and $\beta|_{[\sqrt{\tau_e},\sqrt{t_0-s}]}\subset B_t(x,3r/4)$. Integrating the above inequality, we have

\begin{align*}
|\beta'(\sigma)|^2_{g_{t-\sigma^2}}\ge c(r,T)\left|\beta'(\sqrt{t_0-s})\right|^2_{g_s}-C(r,T)
\end{align*}
for all $\sigma\in[\sqrt{\tau_e},\sqrt{t_0-s}]$.

On the other hand, by (1) and the claim, we have 
\begin{align*}
C(r,L,T) \ge &\ \mathcal{L}(\gamma) -\mathcal{L} (\gamma|_{[0,\tau_e]}) 
\\
=&\ \int_{\sqrt{\tau_e}}^{\sqrt{t_0-s}}\left(\frac{1}{2}|\beta'(\sigma)|^2_{g_{t_0-\sigma^2}} +2\sigma^2 R_{g_{t_0-\sigma^2}}(\beta(\sigma))\right)\,d\sigma
\\
\ge &\ c(r,T)\epsilon\left|\beta'(\sqrt{t_0-s})\right|^2_{g_s}-C(r,T),
\end{align*}
where $\epsilon=\epsilon(r,L,T)>0$ is from the claim. Thus, by Lemma \ref{grad_ll}, we have
\begin{align*}
|\nabla\ell|(y,s)=|\gamma'(t_0-s)|=\frac{|\beta'(\sqrt{t_0-s})|}{2\sqrt{t_0-s}} \le C(r,L,T).
\end{align*}

Finally, by \eqref{eq_l_1}, the curvature assumption (b), and part (1), we also have
\begin{align*}
\left|\frac{\partial}{\partial s}\ell\right|(y,s)\le C(r,L,T).
\end{align*}
This finishes the proof of part (2).
\end{proof}

	\begin{Corollary}\label{Coro:local estimates}
	If
\begin{align*}
\epsilon<\overline{\epsilon}(n),\quad r>0, \quad L< \infty,\quad T>r^2, \quad C\ge \underline{C}(r,L,T),
\end{align*}	
	then the following holds.
	
		Let $(M^n,g_t)_{t\in I}$ be a complete Ricci flow with bounded curvature within each compact time interval and let $\ell:=\tl_{x_0,t_0}$ be (the forward-time version of) the reduced distance based at a fixed space-time point $(x_0,t_0)\in M\times I$; see \eqref{eq:def of tl}. Assume that for a space-time point $(x,t)\in M\times I$, it holds that
		\begin{enumerate}[(a)]
		\item $[t-r^2,t+r^2]\subset I$ and $ r^2<  t_0-t \le T$;
		\item $\displaystyle |\Rm| \le r^{-2}$ on $B_{g_t}(x,r)\times[t-r^2,t+r^2]$;
		\item there is a point $y \in B_{g_{t+r^2}}(x,\epsilon r)$, such that $\ell(y,t+r^2)\le L$.
		\end{enumerate}
Then we have 
\begin{align*}
|\ell|\le C,\quad \left|\tfrac{\partial}{\partial s}\ell\right|+|\nabla\ell|\le C \qquad \text{ on }\qquad B_{g_t}(x,\epsilon r)\times[t-(\epsilon r)^2,t+(\epsilon r)^2].
\end{align*}
	\end{Corollary}
	
	\begin{proof}

By Proposition \ref{l-estiamtes}	 and the distance distortion estimates, we can find $c(n) <1/4$, such that, under the assumptions (a)---(c), if $\epsilon<c(n)$, then the conclusion holds on $B_{g_{t+r^2}}(x,c(n)r)\times$ $[t+(1-2c(n))r^2,(1-c(n))r^2]$. Then the conclusion also holds on 	$B_{g_{t+(1-c(n))r^2}}(x,c(n)r)\times$ $[t+(1-3c(n))r^2,(1-2c(n))r^2]$, and, furthermore,  on	$B_{g_{t+(1-2c(n))r^2}}(x,c(n)r) \times$ $ [t+(1-4c(n))r^2,(1-3c(n))r^2]$, etc., with possibly different constant $C$. After finitely many steps, we obtain the corollary by choosing appropriate $\overline{\epsilon}$.
	
	\end{proof}

\section{The proof of Theorem \ref{main_thm}}

In this section, we present the proof of Theorem \ref{main_thm} by using Bamler's technique in \cite{Bam20c} and the estimates in the previous section. Corollary \ref{coro_main} will follow in a similar fashion. We shall split the proof into several steps.

\subsection{Setting up the contradictory sequence}

Arguing by contradiction, assume that Theorem \ref{main_thm} is false, then we can find a sequence of counterexamples, namely, a sequence of $n$-dimensional Ricci flows $\{(M^i,g^i_t)_{t\in I_i}\}_{i=1}^\infty$, each with bounded curvature within compact time intervals, a sequence of points $(x_i,t_i)\in M^i\times I_i$, a sequence of scales $r_i>0$ with $[t_i-2r_i^2,t_i]\subset I_i$, and a sequence of positive numbers $\epsilon_i\searrow 0$, satisfying
\begin{align*}
\mathcal{V}_{x_i,t_i}(r_i^2)\ge 1-\epsilon_i,
\\\nonumber
r_{\Rm,i}(x_i,t_i) < r_i.
\end{align*}

First of all, we shall perform a normalization for the sequence. Letting
\begin{align*}
\overline{r}_i=r_{\Rm,i}(x_i,t_i)<r_i,
\end{align*}
we consider the scaled flows instead:
\begin{align*}
\tfrac{1}{\overline{r}^{2}_i}g^i_{t_i+\overline{r}_i^2t}.
\end{align*}
These new flows will still be name $g^i_t$ to avoid notational complexity. In this way, we have obtained a sequence of Ricci flows $\{(M^i,g^i_t)_{t\in[-T_i,0]}\}_{i=1}^\infty$ with
\begin{align}\label{eq:contradictory_sequence_original}
\mathcal{V}_{x_i,0}(T_i/2)\ge 1-\epsilon_i,
\\\nonumber
r_{\Rm}(x_i,0)=1,
\end{align}
where we have defined
\begin{align*}
T_i:=\frac{2r_i^2}{\overline{r}_i^2}>2.
\end{align*}

Indeed, the normalized sequence is noncollapsed.

\begin{Lemma}
We have 
\begin{align}\label{eq:lowerboundofN}
\mathcal{N}_{x_i,0}(1) \ge -Y(n)\quad \text{ for all }\quad i.
\end{align}
\end{Lemma}

\begin{proof}
By Perelman's no local collapsing theorem \cite[\S 7.3]{P1} (see also \cite[Theorem 8.1]{MT}), we have
\begin{align*}
\Vol_{g^i_0}\big(B_{g^i_0}(x_i,1)\big)\ge c(n) \quad \text{ for all }\quad i.
\end{align*}
On the other hand, by the standard maximum principle for scalars, we have
\begin{align*}
R_{g^i} \ge -\frac{n}{2} \quad \text{ on }\quad M^i\times[-1,0]
\end{align*}
for all $i$. The lemma then follows from \cite[Theorem 8.1]{Bam20a}.
\end{proof}

Now, we can apply the theory of \cite{Bam20b} to extract a (not relabeled) subsequence from $\{(M^i,g^i_t)_{t\in[-T_i,0]}\}$, such that
\begin{gather}\label{eq:convergence of contra sequence}
	((M^i,g^i_t)_{t\in [-T_i,0]}, (\nu^i_t)_{t\in[-T_i,0]})\xrightarrow[\quad i\to\infty\quad ]{\mathbb{F}}(\mathcal{X},(\nu_t)_{t<0}),
\end{gather}
where $(\mathcal{X},(\nu_t)_{t<0})$ is a metric flow pair defined on $(-T_\infty,0)$, where $T_\infty\ge 2$, and 
\begin{align*}
d\nu^i_t = K(x_i,0\,|\,\cdot, t)\,dg^i_t= (4\pi|t|)^{-\frac{n}{2}}e^{-f_i(\cdot,t)}\,dg^i_t
\end{align*}
is the conjugate heat kernel based at $(x_i,0)$.

The limit metric flow $\mathcal{X}$ admits a decomposition $\mathcal{X}=\mathcal{R}\cupdot\mathcal{S}$. We shall denote by 
\begin{align*}
    K_t:= &\ \mathcal{X}_t\cap K,
    \\
    K_{[t_1,t_2]}:= &\ \mathcal{X}_{[t_1,t_2]}\cap K,
\end{align*}
the time-slice and time-slab of a set $K\subset \mathcal{X}$, respectively. We remark here that Theorem \ref{Regularpart-basic-properties} and Theorem \ref{Thm:local uniform convergence} can be applied to the limit flow $(\mathcal{X},(\nu_t)_{t<0})$ and to the convergence \eqref{eq:convergence of contra sequence}. In particular,
\begin{align*}
d\nu_t= (4\pi|t|)e^{-f}\,dg_t\quad\text{ on }\quad \mathcal{R},
\end{align*}
where $g_t$ is the Riemannian metric on $\mathcal{R_t}$, and $(4\pi|t|)e^{-f}$ is a positive solution to the conjugate heat equation.

We define
\begin{align*}
\ell_i:=\tl_{x_i,0}.
\end{align*}
Note that $\tl_{x_i,0}$ is (the forward time version of) the reduced distance \eqref{eq:def of tl} based at $(x_i,0)$. Our goal is to analyse the possible convergence of $\ell_i$ under the setting of Theorem \ref{Thm:local uniform convergence}, and use Perelman's monotonicity formulas to show that the limit flow $(\mathcal{X},(\nu_t)_{t<0})$ is a Gaussian shrinker. To this end, some local uniform estimates for $\ell_i$ are needed, and this is the object of the next subsection.

Lastly, we recall the following useful estimate due to Bamler.

\begin{Proposition}[Bamler's $L^p$ estimates] \label{Prop:Bamler Lp}
For any $T\in(0,1/2]$, there is a $C(T,Y,n)>0$ independent of $i$, such that
\begin{align*}
\int_{-2T}^{-T}\int_{M_i}\left(|\Delta_{g^i_t}f_i|^2+|\nabla f_i|_{g^i_t}^4\right)\,d\nu^i_tdt \le C(T,Y,n)
\end{align*}
for all $i$. Here $Y$ is the entropy bound in \eqref{eq:lowerboundofN}.
\end{Proposition}

\begin{proof}
The standard maximum principle implies that
$$R_{g^i}\ge -n\quad \text{ on }\quad M^i\times [-3/2,0]$$
holds for each $i$. Thus, combining \eqref{eq:lowerboundofN} and \cite[Proposition 5.2]{Bam20a}, we have
\begin{align*}
    \mathcal{N}_{x_i,0}(3/2)\ge -Y_1(n)\quad \text{ for all }\quad i.
\end{align*}
The proposition follows from \cite[Proposition 6.2]{Bam20c} (it is easy to see from the proof that, the assumption of \cite[Proposition 6.2]{Bam20c} can be relaxed to $\mathcal{N}_{x,0}(3r^2/2)\ge -Y$).
\end{proof}

\subsection{Local uniform $C^{0,1}$-estimates} In this subsection, we shall prove the following local uniform estimates for $\ell_i$.

\begin{Proposition}\label{Prop:local uniform estimate of l}
For any \emph{compact} set $K\subset\mathcal{R}_{(-1,0)}$, the following holds on $\psi_i(K)$ whenever $i$ is large enough
\begin{align*}
|\ell_i| &\ \le C(K),
\\
\left|\frac{\partial}{\partial t}\ell_i\right|+|\nabla\ell_i|_{g^i_t} &\ \le C(K),
\end{align*}
where $C(K)$ is a constant depending on $K$, and in particular, independent of $i$.
\end{Proposition}

We begin with a useful result of Perelman. 

\begin{Lemma}[{\cite[Corollary 9.5]{P1}}]\label{Lm:Perelman's LYH Harnack}
For each $i$, the following holds on $M^i\times [-T_i,0]$:
\begin{align*}
f_i\le \ell_i\qquad \text{ or }\qquad (4\pi|t|)^{-\frac{n}{2}}e^{-\ell_i} \le (4\pi|t|)^{-\frac{n}{2}}e^{-f_i}.
\end{align*}
\end{Lemma}

\begin{Lemma} \label{Lm:local uniform estimate of l}
For any $\epsilon<\overline{\epsilon}(n)$ and $\mathfrak{y}\in\mathcal{R}_{(-1,0)}$, there is a positive numbers $r>0$ and $L<\infty$ depending on $\mathfrak{y}$ and $\epsilon$ but independent of $i$, such that the following holds for all $i$ large enough. Let $(y_i,t)=\psi_i(\mathfrak{y})$, where $t=\mathfrak{t}(\mathfrak{y})$, then 
\begin{gather*}
|\Rm_{g^i}| \le r^{-2}\ \text{ on }\ B_{g^i_t}(y_i,r)\times [t-r^2,t+r^2],
\end{gather*}
and there is a point $z_i\in B_{g^i_{t+r^2}}(y_i,\epsilon r)$, such that 
\begin{align*}
\ell_i(z_i,t)\le L. 
\end{align*}
\end{Lemma}

\begin{proof}

Since $\mathfrak{y}\in\mathcal{R}_{(-1,0)}$ is a regular point, we can find an unscathed space-time-product neighbourhood  with radius $2r>0$, such that
\begin{align*}
|\Rm_g|\le (2r)^{-2}\qquad \text{ on }\qquad B_t(\mathfrak{y},2r)\times [t-4r^2,t+4r^2]\subset \mathcal{R}_{(-1,0)}.
\end{align*}

By the local smooth convergence, we must have 
\begin{align*}
B_{g^i_t}(y_i,r)\subset \psi_i\big(B_t(\mathfrak{y},2r)\big)
\end{align*}
and
\begin{align*}
|\Rm_{g^i}|\le r^{-2}\quad \text{ on }\quad B_{g^i_t}(y_i,r)\times [t-2r^2,t+2r^2].
\end{align*}
whenever $i$ is large enough. 

Letting $t'=t+r^2$, another consequence of the local smooth convergence is 
\begin{align*}
\lim_{i\to\infty} \int_{B_{g^i_{t'}}(y_i,\epsilon r)}(4\pi |t'|)^{-\frac{n}{2}} e^{-f_i}\,dg^i_{t'} = &\ \int_{B_{t'}(\mathfrak{y}(t'),\epsilon r)}(4\pi \tau)^{-\frac{n}{2}} e^{-f}\,dg_{t'} 
\\
= &\ \nu_{t'} \big(B_{t'}(\mathfrak{y}(t'),\epsilon r)\big)
\\
:= &\ 2c>0,
\end{align*}
where $\mathfrak{y}(\cdot)$ stands for the world-line translation and we have applied the total support property of $\mathcal{X}$ (Theorem \ref{Regularpart-basic-properties}(c)). Thus,
\begin{align*}
\int_{B_{g^i_{t'}}(y_i,\epsilon r)}(4\pi |t'|)^{-\frac{n}{2}} e^{-f_i}\,dg^i_{t'} >c\qquad \text{ for all $i$ large enough}.
\end{align*}

On the other hand, by Lemma \ref{Lm:Perelman's LYH Harnack} and assumption \eqref{eq:contradictory_sequence_original}, we have
\begin{align*}
\int_{B_{g^i_{t'}}(y_i,\epsilon r)}(4\pi |t'|)^{-\frac{n}{2}} (e^{-f_i}-e^{-\ell_i})\,dg^i_{t'} \le &\ \int_{M^i}(4\pi |t'|)^{-\frac{n}{2}} (e^{-f_i}-e^{-\ell_i})\,dg^i_{t'} 
\\
= &\ 1- \mathcal{V}_{x_i,0}(|t'|) <\epsilon_i.
\end{align*}
Thus, 
\begin{align*}
\int_{B_{g^i_{t'}}(y_i,\epsilon r)}(4\pi |t'|)^{-\frac{n}{2}} e^{-\ell_i}\,dg^i_{t'} \ge c-\epsilon_i >c/2
\end{align*}
for all $i$ large enough. On the other hand, since
\begin{align*}
\lim_{i\to\infty}\operatorname{Vol}_{g^i_{t'}}\big(B_{g^i_{t'}}(y_i,\epsilon r)\big) = \operatorname{Vol}_{g_{t'}}\big(B_{g_{t'}}(\mathfrak{y}(t'),\epsilon r)\big),
\end{align*}
we have
\begin{align*}
\operatorname{Vol}_{g^i_{t'}}\big(B_{g^i_{t'}}(y_i,\epsilon r)\big) \le 2\operatorname{Vol}_{g_{t'}}\big(B_{g_{t'}}(\mathfrak{y}(t'),\epsilon r)\big):=C
\end{align*}
for all $i$ large enough. Thus, 
\begin{align*}
c/2 < \int_{B_{g^i_{t'}}(y_i,\epsilon r)}(4\pi |t'|)^{-\frac{n}{2}} e^{-\ell_i}\,dg^i_{t'} \le C\sup_{B_{g^i_{t'}}(y_i,\epsilon r)}(4\pi |t'|)^{-\frac{n}{2}} e^{-\ell_i}.
\end{align*}
Then we obtain the lemma.

\end{proof}

\begin{proof}[Proof of Proposition \ref{Prop:local uniform estimate of l}]
Combing Lemma \ref{Lm:local uniform estimate of l} and Corollary \ref{Coro:local estimates}, the proposition then follows from a simple covering argument.
\end{proof}

\subsection{Soliton structure on the limit} In this subsection, we prove the following proposition.

\begin{Proposition}\label{Prop:soliton structure}
The limit metric flow $(\mathcal{X},(\nu_t)_{t<0})$ in \eqref{eq:convergence of contra sequence} satisfies
\begin{align*}
\mathcal{N}(\tau):= \int_{\mathcal{R}_{-\tau}} f\,d\nu_{-\tau}-\frac{n}{2}\equiv 0\quad \text{ for all }\quad \tau\in(0,1).
\end{align*}
Thus, $(\mathcal{X}_{(-1,0)},(\nu_t)_{t\in(-1,0)})$ is the Gaussian shrinker.
\end{Proposition}

According to Theorem \ref{Thm:local uniform convergence} and Proposition \ref{Prop:local uniform estimate of l}, we have that $\ell_i\circ\psi_i$ converges, after passing to a subsequence, to a function $\ell:\mathcal{R}_{(-1,0)}\to\mathbb{R}$ in the $C^{0,\alpha}_{\operatorname{loc}}$ and weak $*W^{1,2}_{\operatorname{loc}}$ senses.

\begin{Lemma}
$f\equiv\ell$ on $\mathcal{R}_{(-1,0)}$.
\end{Lemma}

\begin{proof}
Let us fix an arbitrary $\tau\in(0,1)$ and let $K\subset \mathcal{R}_{-\tau}$ be a compact set. Then by Lemma \ref{Lm:Perelman's LYH Harnack} and the local convergence, we have
\begin{align*}
0\le \int_{K}(4\pi\tau)^{-\frac{n}{2}}(e^{-f}-e^{-\ell})\,dg_{-\tau} = &\ \lim_{i\to\infty} \int_{\psi_i(K)}(4\pi\tau)^{-\frac{n}{2}}(e^{-f_i}-e^{-\ell_i})\,dg^i_{-\tau}
\\
\le &\ \lim_{i\to\infty} \int_{M^i}(4\pi\tau)^{-\frac{n}{2}}(e^{-f_i}-e^{-\ell_i})\,dg^i_{-\tau}
\\
= &\ \lim_{i\to\infty} (1-\mathcal{V}_{x_i,0}(\tau))
\\
= &\ 0,
\end{align*}
by the assumption \eqref{eq:contradictory_sequence_original}. The lemma follows immediately.
\end{proof}

\begin{Lemma}
The following holds on $\mathcal{R}_{(-1,0)}$:
\begin{align}\label{eq:identity for f}
-2\frac{\partial f}{\partial t} +|\nabla f|^2_{g_t} - R_{g_t}+\frac{f}{\tau} = &\ 0,
\\ \nonumber
2\Delta_{g_t }f -|\nabla f|^2_{g_t}+R_{g_t}+\frac{f-n}{\tau}= &\ 0,
\end{align}
where $\tau=-t=|t|$.
\end{Lemma}

\begin{proof}
Since each $\ell_{x_i,0}$ satisfies \eqref{eq_l_1} in the sense of distribution, and the functions $\ell_i\circ\psi_i=\tl_{x_i,0}\circ\psi_i$ converge to $\ell$ in the $C^{0,\alpha}_{\operatorname{loc}}$  and  weak $*W^{1,2}_{\operatorname{loc}}$ senses on $\mathcal{R}_{(-1,0)}$, we have that $\ell$ (hence $f$) satisfies the first equation of \eqref{eq:identity for f} in the sense of distribution on $\mathcal{R}_{(-1,0)}$. While $f$ is smooth on $\mathcal{R}$, we obtain the first equation in the classical sense. The second equation is a combination of the first with the conjugate heat equation.

\end{proof}

\begin{proof}[Proof of Proposition \ref{Prop:soliton structure}]
By \cite[Theorem 2.18]{Bam20c} and \cite[Theorem 6.3]{CMZ24}, we need only to show that $\mathcal{N}(\tau)\equiv 0$ for all $\tau\in(0,1)$.

 Let us define
\begin{eqnarray*}
	u&:=&(4\pi\tau)^{-\frac{n}{2}}e^{-f},
	\\
	v&:=&\left(\tau(2\Delta_{g_t} f-|\nabla f|^2_{g_t}+R_{g_t})+f-n\right)u,
\end{eqnarray*}
where $\tau=-t=|t|$. Perelman's computation  \cite[Proposition]{P1} shows that
\begin{eqnarray*}
	\left(-\frac{\partial}{\partial t}-\Delta_{g_t}+R_{g_t}\right)v=-2\tau\left|\Ric_{g_t}+\nabla^2 f-\frac{1}{2\tau}{g_t}\right|^2u,
\end{eqnarray*}
holds on the smooth part $\mathcal{R}_{(0,1)}$. Since $v\equiv 0$ and $u>0$, it follows that
\begin{align*}
\Ric+\nabla^2 f-\frac{1}{2\tau} g=0 \qquad \text{ on }\qquad \mathcal{R}_{(-1,0)}.
\end{align*}
Combining the trace of the above equation with the second equation in \eqref{eq:identity for f}, we have
\begin{align}\label{eq:last step to show N=00}
\tau\big(\Delta_{g_t} f -|\nabla f|^2_{g_t}\big) + f-\frac{n}{2} =0 \qquad \text{ on }\qquad \mathcal{R}_{(-1,0)}.
\end{align}

\noindent\underline{\textbf{Claim.}} \emph{For any $T\in(0,1/2)$, we have}
\begin{align*}
\int_{-2T}^{-T}\int_{\mathcal{R}_t}\tau(\Delta_{g_t} f -|\nabla f|^2_{g_t}\big) \,d\nu_tdt = \int_{-2T}^{-T}\int_{\mathcal{R}_t}\tau(\Delta_{g_t} f -|\nabla f|^2_{g_t}\big) \,(4\pi\tau)^{-\frac{n}{2}}e^{-f}\,dg_tdt =0.
\end{align*}

\begin{proof}[Proof of the claim.]
Let us fix any $T\in(0,1/2)$ and let $\delta>0$ be an arbitrary small number. By Theorem \ref{Regularpart-basic-properties}(a), we fix $j$ large enough, so that $U_j':=U_j\cap\mathcal{R}_{[-2T,-T]}$ ( $U_j$ is as in the statement of Theorem \ref{Thm:local uniform convergence}) satisfies that 
\begin{align*}
\int_{-2T}^{-T}\int_{U'_{j,t}}\,d\nu_tdt>T-\delta, \qquad \int_{-2T}^{-T}\int_{\mathcal{R}_t\setminus {U'_{j,t}}}\,d\nu_tdt <\delta.
\end{align*} 

Furthermore, local smooth convergence implies that
\begin{align*}
\int_{-2T}^{-T}\int_{M^i\setminus \psi_i({U'_{j,t}})}\,d\nu^i_tdt = &\ T-\int_{-2T}^{-T}\int_{\psi_i({U'_{j,t}})}\,d\nu^i_tdt < T-(T-2\delta) =2\delta
\end{align*}
whenever $i\gg j$. Thus, 
\begin{align*}
&\ \left|\int_{-2T}^{-T}\int_{{U'_{j,t}}}\tau(\Delta_{g_t} f -|\nabla f|^2_{g_t}\big) \,d\nu_tdt\right|
\\
=&\ \lim_{i\to\infty}\left|\int_{-2T}^{-T}\int_{\psi_i({U'_{j,t}})}\tau(\Delta_{g^i_t} f_i -|\nabla f_i|^2_{g^i_t}\big) \,d\nu^i_tdt\right| 
\\
= &\ \lim_{i\to\infty}\left|\int_{-2T}^{-T}\int_{M^i\setminus\psi_i({U'_{j,t}})}\tau(\Delta_{g^i_t} f_i -|\nabla f_i|^2_{g^i_t}\big) \,d\nu^i_tdt\right|
\\
\le &\  \lim_{i\to\infty}2T\left(\int_{-2T}^{-T}\int_{M^i\setminus\psi_i({U'_{j,t}}t)} \,d\nu^i_t\right)^{\frac{1}{2}} \left(2\int_{-2T}^{-T}\int_{M^i}\left(\left|\Delta_{g^i_t} f_i\right|^2 +|\nabla f_i|^4_{g^i_t}\right) \,d\nu^i_tdt\right)^{\frac{1}{2}}
\\
\le &\ C(Y,T,n)\delta^{\frac{1}{2}},
\end{align*}
where we have applied Bamler's integral estimates (Proposition \ref{Prop:Bamler Lp}). The claim follows from taking $j\to\infty$ and $\delta\to 0$.
\end{proof}

Combining \eqref{eq:last step to show N=00} and the claim above, we have that
\begin{align*}
\int_{T}^{2T}\mathcal{N}(\tau)\,d\tau=0
\end{align*}
for all $T\in(0,1/2)$. Since $\mathcal{N}(\tau)\le 0$ for all $\tau>0$ (\cite[Theorem 2.11]{Bam20c}), we have that $\mathcal{N}(\tau)\equiv 0$ for all $\tau\in(0,1)$. 

\end{proof}

\subsection{Completion of the proof}

\begin{proof}[Proof of Theorem \ref{main_thm}]
By
\S4.1---\S4.3, a contradictory sequence satisfying \eqref{eq:contradictory_sequence_original} converges, after passing to a subsequence, to the Euclidean space in the local-smooth sense over the time interval $(-1,0)$. The definition of local smooth convergence shows that, for any $\delta>0$, the geometry in 
\begin{align*}
B_{g^i_0}(x_i,\delta^{-1})\times[-(1-\delta),-\delta]
\end{align*}
must be very Euclidean-like in the smooth sense, whenever $i$ is large enough. This fact, combined with Perelman's pseudolocality theorem \cite[Theorem 10.1]{P1} and Bamler's backward pseudolocality theorem \cite[Theorem 2.47]{Bam20c}, shows that the curvature on 
\begin{align*}
B_{g^i_0}(x_i,\delta^{-1})\times[-1,0]
\end{align*}
also converges to $0$ uniformly. This is a contradiction to \eqref{eq:contradictory_sequence_original}.
\end{proof}

\begin{proof}[Proof of Corollary \ref{coro_main}]
We implement a similar argument of contradiction as in the proof of Theorem \ref{main_thm}. Assume for some $\epsilon>0$ the Corollary fails, then we can find a sequence of counterexamples, which, after proper scaling and shifting of time, satisfies the following properties: 
\begin{enumerate}
\item $\{(M^i,g^i_t)_{t\in[-T_i,0]}\}_{i=1}^\infty$ is a sequence of Ricci flows, each with bounded curvature, and $T_i\ge 1$;
\item $\mathcal{V}_{x_i,0}(T_i)\to 1$;
\item $\operatorname{Vol}_{g^i_0}\big(B_{g^i_0}(x_i,r_i')\big) \le (1-\epsilon)\omega_n (r_i')^n$ for some $r_i'\in(0,\epsilon^{-1})$.
\end{enumerate}
By the same argument as in the proof of Theorem \ref{main_thm}, we have $\{(M^i,g^i_t)_{t\in[-1/2,0]}\}_{i=1}^\infty$ converges to the static Euclidean space in the Cheeger-Gromov-Hamilton sense. This contradicts (3) above.
\end{proof}

\end{document}